\documentclass[12pt]{article}
\usepackage{amsfonts,amsmath,amssymb}

\newtheorem{theor}{Theorem}[section]
\newtheorem{lemm}[theor]{Lemma}
\newtheorem{corollary}[theor]{Corollary}
\newtheorem{proposition}[theor]{Proposition}

\newenvironment{proof}{{\em Proof.}}{\hspace*{\fill}$\Box$\par\vspace{4mm}}

\newtheorem{example}[theor]{Example}
\newtheorem{remark}[theor]{Remark}

\def\C{\mathcal C}

\def\D{\mathcal D}
\def\K{\mathcal K}
\def\Kn{\mathop{\rm Kn}}
\def\I{I}
\def\J{J}
\def\L{\lambda}
\def\Lmin{\lambda}
\def\Lmax{\lambda_1}

\def\LminC{\lambda^*_{\scriptscriptstyle \rm C}}
\def\LminK{\lambda^*_{\scriptscriptstyle \rm K}}

\def\Cay{\rm Cay}

\def\x{\times}

\def\vform{\bf }
\def\ones{{\vform 1}}
\def\zeros{{\vform 0}}

\def\vw{{\vform w}}
\def\vx{{\vform x}}

\def\vz{{\vform z }}

\begin{document}

\title{Some observations on the smallest adjacency eigenvalue of a graph}

\author{
Sebastian M. Cioab\u{a}\footnote{Department of Mathematical Sciences, University of Delaware, Newark, DE 19716-2553, USA, {\tt cioaba@udel.edu}. This research has been partially supported by NSERC, NSF grants DMS-1600768 and CIF-1815922 and a and a JSPS Invitational Fellowship for Research in Japan S19016.}
\,, 
Randall J. Elzinga\footnote{Akira Health, London, Ontario, Canada, \tt{rjelzinga@gmail.com}.}
~and
 David A. Gregory\footnote{Department of Mathematics, Queen's University at Kingston, Canada. David Gregory passed away on July 12, 2013. Part of this work was done with David while the first two authors were his graduate students in 2005.}
}

\date{\today}
\maketitle
\begin{abstract}
In this paper, we discuss various connections between the smallest eigenvalue of the adjacency matrix of a graph and its structure. There are several techniques for obtaining upper bounds on the smallest eigenvalue, and some of them are based on Rayleigh quotients, Cauchy interlacing using induced subgraphs, and Haemers interlacing with vertex partitions and quotient matrices. In this paper, we are interested in obtaining lower bounds for the smallest eigenvalue. Motivated by results on line graphs and generalized line graphs, we show how graph decompositions can be used to obtain such lower bounds.
\end{abstract}

\begin{center}{\em Dedicated to the memory of Slobodan K. Simi\'{c}}
\end{center}

\section{Introduction}

Our graph notation is standard, see \cite{BH} for undefined terms or notation. The {\em eigenvalues} of a graph $G=(V,E)$ are the eigenvalues of its adjacency matrix $A=A(G)$. For a graph $G$ with $n$ vertices and $\ell\geq 1$, denote by $\lambda_{\ell}(G)$ the $\ell$-th greatest eigenvalue of $G$ and let $\lambda^{\ell}(G)=\lambda_{n-\ell+1}(G)$ be its $l$-th smallest eigenvalue. Let $\Lmin(G)$ denote the smallest eigenvalue $\lambda^1(G)$. The smallest eigenvalue of a graph is closely related to its chromatic number and independence number \cite{BH,GM}. Since the spectrum of a connected graph is symmetric if and only if the graph is bipartite, it is natural to think of $\Lmin(G)$ as a measure of how bipartite $G$ is. It is therefore not surprising that the smallest eigenvalue has close connections to the max-cut \cite{AS,BCIM,GW,K}. There are several methods to obtain upper bounds for $\Lmin(G)$. Using Rayleigh quotients, it is well known that
\begin{equation}
\large{\Lmin(G)=\min_{\vx\neq 0}\frac{\vx^TA\vx}{\vx^T\vx}.}
\end{equation}
Depending on the context, choosing appropriate vectors can yield useful upper bounds on $\Lmin(G)$ such as the ones involving the max-cut \cite{AS,GW} or Hoffman's ratio bound on the independence number (see \cite[Section 3.5]{BH} for example). The connection between eigenvalues and Rayleigh quotients also yields important interlacing results such as Cauchy interlacing or Haemers interlacing \cite[Section 3.5]{BH}. In each case, the eigenvalues of a smaller matrix (principal submatrix of $A$ in the case of Cauchy interlacing or quotient matrix in the case of Haemers interlacing) interlace the eigenvalues of $A$ and therefore, $\Lmin(G)$ is bounded from above by the smallest eigenvalue of this smaller matrix. Again, these important methods yield interesting consequences such Cvetkovic's inertia bound for the independence number \cite[Theorem 3.5.1]{BH} or Hoffman's ratio bound for the chromatic number \cite[Theorem 3.6.2]{BH} to name just a few. In other situations, manipulations of the trace of powers of the adjacency matrix of a graph (see \cite{C1} for example) or edge perturbations in graphs (see \cite{BCRS1,BCRS2}) can yield upper bounds for $\Lmin$.

In this paper, we are interested in the finding lower bounds for the smallest eigenvalue $\Lmin(G)$ using graph decompositions. We apply our methods to various situations and we describe their successes and limitations. In general, lower bounds on the smallest eigenvalue of a graph are not easy to obtain. In \cite{AS}, Alon and Sudakov show that $\Lmin\geq -\Delta + \frac{1}{(D+1)n}$ for a nonbipartite simple graph with maximum degree $\Delta$ and diameter $D$ (see also \cite{C2,Nik} for small improvements). Trevisan \cite{Tre} obtained interesting connections between $\Lmin(G)$ and the  bipartiteness ratio $\beta(G)$ which is defined as $\min_{S\subset V:S=L\cup R}\frac{2e(L)+2e(R)+e(S,V\setminus S)}{\sum_{v\in S}d(v)}$, where the minimum is taken over all subsets $S$ of $V$ and all partitions $L\cup R$ of $S$, $e(L)$ denotes the number of edges in the subgraph induced by $L$ (similar definition for $e(R)$) and $e(S,V\setminus S)$ denotes the number of edges with exactly one endpoint in $S$. Trevisan's results are similar to the ones relating the second largest eigenvalue of a graph to its expansion/isoperimetric constant (see \cite{Alon,AM,Tre}) and for a $d$-regular graph, give the following interesting lower bound: $\Lmin(G)\geq -d+\frac{\beta^2}{d}$. However, we have not been able to use this bound for the graphs considered in this paper as the parameter $\beta$ does not seem easy to calculate.

In Section \ref{sec:gradec}, we use weighted graph decompositions of the edge set of a graph to bound the spectrum of a graph from below.
Our results are similar and have been obtained independently from the recent work of Knox and Mohar \cite{KM}. 

In Section \ref{sec:clique}, we specialize these decompositions to clique decompositions and give examples when the bounds are tight and when they are not. It is not surprising that for line graphs, generalized line graphs and point-line graphs of finite geometries, our bounds are tight, but there are many graphs where our methods do not yield tight bounds.

In Section \ref{sec:K1kfree}, we discuss the smallest eigenvalue of $K_{1,k}$-free graphs. Linial \cite{Lin} asked whether the property of the eigenvalues of line graphs to be bounded from below by an absolute constant also holds for claw-free simple graphs. In \cite{C0}, the first author showed that the answer is negative by describing a family of regular claw-free simple graphs with arbitrarily negative eigenvalues. Recently, motivated by problems in topological combinatorics, Aharoni, Alon and Berger \cite{AAB} studied the largest eigenvalue of the Laplacian of $K_{1,k}$-free graphs which when restricted to regular graphs, is equivalent to studying the smallest eigenvalue of regular $K_{1,k}$-free graphs. In Section \ref{sec:K1kfree}, we describe their results and remark that their proof actually gives a more general lower bound for the smallest eigenvalue of graphs with dense neighborhoods. In \cite{AAB}, the authors also constructed examples of $d$-regular $K_{1,k}$-free graphs with very negative $\Lmin$ by taking clique blowups of bipartite $(k-1)$-regular graphs. Their construction works when $d=ks-1$ for $k\geq 3$ and $s\geq 2$. In the case of claw-free graphs ($k=3$), their construction works for $d=3s-1$ and $s\geq 2$. In this section, we also show that every cubic claw-free graph has $\Lmin\geq -2.272$ which slightly improves the lower bound of $-2.5$ from \cite{AAB}.

\section{Smallest eigenvalue and graph decompositions}\label{sec:gradec}

In this section, we introduce graph decompositions as a means of obtaining lower bounds on the least adjacency eigenvalue of a graph. Like many such general bounds, there are cases where the estimates are strong and others where they are weak. The advantage of using decompositions lies in the flexibility of choice: for graphs with many triangles, decompositions by complete graphs are often successful,
while for graphs with few triangles, decompositions allowing paths and cycles may be more fruitful. Also, taking decompositions of an odd power of a graph can sometimes improve a lower bound.

We begin with a simple observation on matrix decompositions and then
introduce weighted graphs to capture the generality of an arbitrary symmetric matrix.
All matrices are assumed to be symmetric and have real entries.

A {\em decomposition} of a matrix $M$ of order $n$ is  a collection
$M_1, \ldots M_m$ of matrices such that
$$ M= M_1 + \cdots + M_m.$$
Using Rayleigh quotients, we quickly obtain a lower bound on the least eigenvalue of $M$:
\begin{equation} \label{trivbnd}
\Lmin(M) \ge \Lmin(M_1) + \cdots + \Lmin(M_m).
\end{equation}
The {\em support} of a matrix $M$ is the set of all row indices $i$ of $M$
such that $M_{ij}\ne 0$ for some column index $j$ of $M$.
Arbitrary real symmetric matrices may be regarded as adjacency matrices of edge-weighted graphs and,
in our application to graph decompositions, often have small support.
In such cases, the simple estimate (\ref{trivbnd}) can be improved.
We require the following notation.

A {\em weighted} graph $H=H(w)=H(V,w)$ is
a graph $H$ with vertex set $V$ together with a function $w$ that assigns a (possibly negative)
real number $w(uv)$ to each unordered vertex pair $uv$.  Also, $w(uv)=0$ if $u\not\sim v$.
This implies that $w(uu)=0$ if there is no loop on $u$ in $H$.

We say that a set $\D=\{H^j(V_j,w_j) : 1\leq j\leq m \}$ of
weighted graphs is a {\em decomposition} of
a weighted graph $H(V,w)$  and write
$$ H = H^1 +H^2 + \cdots +H^m$$
if $V_j\subseteq V$ for $1\leq j\leq m$ and
$w(uv)=\sum_{j=1}^{n} w_j(uv)$ for all unordered pairs of vertices $u,v$ of $V$.  Here,
we take $w_j(uv)=0$ if either $u$ or $v$ is not in $V_j$.
The (weighted) adjacency matrix of a weighted graph $H(V,w)$ is the symmetric matrix $A(w)$, indexed
by the vertices in $V$, with $u,v$ entry equal to $w(uv)$.
Let $\D_u$ denote the set of graphs in $\D$
that contain vertex $u$.

\begin{theor}\label{Dthm} Let
$\D=\{H^j(V_j,w_j) :1\leq j\leq m\}$  be a decomposition of a weighted graph $H=H(V,w)$.
For each vertex $u\in V$, let $\Lmin(\D_u)$ be the sum
of the minimum eigenvalues of the graphs of $\D$ that contain vertex $u$.
Then
\begin{equation} \label{Dineq}
\Lmin(H)\geq \min_u\Lmin(\D_u).
\end{equation}
Let $\Lmin(\D)=\min_u\Lmin(\D_u)$.  Then
equality holds in (\ref{Dineq}) if and only if
there is a vector $\vx\ne \zeros$ of real numbers indexed by the vertices in $V$ such that
\begin{enumerate}
\item $x_u=0$ whenever
$\Lmin(\D)<\Lmin(\D_u)$; and,
\item For each $1\leq j\leq m$, the restriction  $\vx_j$ of $\vx$ to $V_j$ is either a zero vector or
an eigenvector of $H^j$ with eigenvalue $\Lmin(H^j)$.
\end{enumerate}
\end{theor}
\begin{proof} Let $M=A$ be the adjacency matrix of $H(V,w)$ and, for
$1\leq j\leq m$,  let $M_j$ be the adjacency matrix $A_j$ of $H^j(V_j,w_j)$ augmented by zero rows and
columns indexed by vertices of $V$ not in $V_j$.
Because $H = H^1 +H^2 + \cdots +H^m$,  we have $M=M_1+M_2+\cdots+M_m$.
Let $E_j$ be the (0,1)-diagonal matrix with $(u,u)$
entry equal to 1 if vertex $u$ is in $V_j$ and 0, otherwise. The matrix
$R=-\sum_{j=1}^m\Lmin(H^j)E_j$
is a diagonal matrix with $(u,u)$ entry $r_u=-\Lmin(\D_u)$. Let
$B=\sum_{j=1}^m(M_j-\Lmin(H^j)E_j)$. Then $B$ is positive
semidefinite because each of its summands is. Therefore,
letting $r=\max_u r_u= -\Lmin(\D)$,
it follows that the matrix
$$P=M + rI = B + (rI-R)$$
is positive semidefinite. Thus,
$\Lmin(M)+r =\Lmin(M+rI) =\Lmin(P) \ge 0$ with equality
if and only if $\zeros$ is an eigenvector of $P$;
equivalently, $\Lmin(H)\ge \Lmin(\D)$ with equality if and only if $P\vx=\zeros$
for some vector $\vx\ne\zeros$.

Since $P$ is a positive semidefinite matrix, $P=N^TN$ for some matrix $N$.
Thus $P\vx=\zeros$ if and only if
$\vx^TP\vx=||N\vx||^2=0$.  Substituting for $P$ and recalling that each of its summands
is positive semidefinite, it follows that $P\vx=\zeros$
if and only if $\vx^TR\vx=\sum_u (r-r_u)x_u^2=0$  and, for each $1\leq j\leq m$,
$\vx_j^T(M_j-\Lmin(H^j)E_j)\vx_j=0$.
The first condition holds if and only if $x_u=0$ whenever $r_u<r$.
Because the matrices involved are positive semidefinite,
the second condition holds if and only if
$(A_j-\Lmin(H^j)I)\vx_j=0$ for $1\leq j\leq m$.
\end{proof}

The {\em Cartesian product} $G=G^1\Box \cdots \Box G^m$ of the simple graphs $G^i=(V_i,E_i),1\leq i\leq m$,
is the simple graph on the vertex set $V=V_1 \x \cdots \x V_m$ with two vertices $u, v\in V$
adjacent if there is an index $j$ such that $u_j\sim v_j$ in $G^j$ and $u_i=v_i$ for all $i\ne j$.
Using the Hadamard product $\otimes$, we state the well-known fact that the adjacency matrix $A(G)$ is the sum of the $m$ products of the form $I\otimes I \otimes \cdots \otimes A(G^j)\otimes \cdots \otimes I$, $1\leq j\leq m$  where the $j$-th term is a product of
$m-1$ identity matrices (of orders equal to those of the corresponding graphs) together with $A(G^j)$ in the $j$-th position.
From this it follows that if $\vx_i$ is an $\alpha_i$-eigenvector of $G^i$, $1\leq i\leq m$,
then $\vx_1\otimes \cdots \otimes \vx_m$ is an $\alpha_1+ \cdots + \alpha_m$ eigenvector of $A(G)$.
As an illustration, in the following example,
we also obtain the least eigenvalue $\Lmin(G^1)+ \cdots + \Lmin(G^m)$ of $G$ using Theorem \ref{Dthm}.

\begin{example} {\rm  ({\em Cartesian products and Hamming graphs})
Let $G=G^1\Box \cdots \Box G^m$.
Taking the induced subgraphs of $G$ on sets of vertices where
all but one of the coordinates is fixed, we obtain a decomposition
$\D$  consisting of copies of $G^i$, $1\leq i\leq m$.  Since each
vertex $u$ of $G$ is contained in one copy of $G^i$ for each $1\leq i\leq m$,
we have
$\Lmin(G)\geq \Lmin(G^1)+ \cdots + \Lmin(G^m)$ by (\ref{Dineq}).
However, finding the eigenvector that satisfies the conditions
sufficient for equality in Theorem \ref{Dthm} would be difficult
without appealing to the above form of the adjacency matrix for $G$.

A Cartesian product of complete graphs, each of order at least 2, is called a {\em Hamming graph}.
Thus, if $G$ is a Hamming graph whose vertices are $m$-tuples, then $\Lmin(G)=-m$.
}\end{example}

We frequently require the following graphs in decompositions.
The {\em loop graph} of order $n$ has a loop at each vertex and no other edges.
We use the symbol $I_n$  for both it
and its adjacency matrix (the identity matrix of order $n$).
The {\em looped complete} graph of
order $n$ is obtained by adding a loop to each vertex of $K_n$.
We use the symbol $J_n$ for both it and its adjacency matrix (the all-one matrix).
The {\em simple complete graph} of order $n$ has an edge
between each pair of {\em distinct} vertices.  We use the symbol $K_n$
to denote both it and its adjacency matrix, $J_n-I_n$.
The graph $K_1$ has no edges and is not used in decompositions.
The graph
$J_1=I_1$ is called a {\em single loop} and may appear in decompositions.
Note that although
$\Lmin(\J_n)=0$ when $n>1$, we have $\Lmin(\J_1)=1$.

For a graph $G$ and real number $c$ (possibly negative),
we write $cG$ for the weighted graph
that has constant weight function $c$ on the edges (and 0 on the non-edges).
In particular, each graph $G=1G$ may be regarded as a weighted graph with edge weights all $1$
while $-G$ is the graph $G$ with edge weights all $-1$.
Thus, $\Lmin(I_n)=1$ and $\Lmin(-I_n)=-1$.  For $n>1$, $\Lmin(J_n)=0$, but $\Lmin(J_1)=\Lmin(I_1)=1$. Also,
$\Lmin(-J_n)=-n$ for all $n\ge 1$.  Because $n>1$ for $K_n$ and $K_n=J_n-I_n$, we have $\Lmin(K_n)=-1$ while $\Lmin(-K_n)= -(n-1)$.

The next three examples use a multigraph $\hat G = G^{(k)}$ formed from a simple graph $G$; that is,
a multigraph  with adjacency matrix $A(G)^k$ where $k$ is a positive integer.
Thus, the number $w(uv)$ of edges in $G^{(k)}$ with endpoints $u,v$
is the number of $uv$-walks of length $k$ in $G$.
For example, if $G$ is a simple graph and $u,v$ are adjacent vertices in $G$
with degrees $d(u), d(v)$ and neighbour sets $N(u), N(v)$ then the number of
edges in $G^{(3)}$ with endpoints $u,v$ equals $d(u)+d(v)-1$ if $u\sim v$
and equals the number of edges between $N(u)$ and $N(v)$ if $u \not\sim v$.
If $k$ is odd, the least eigenvalues of $G$ and $G^{(k)}$ are related by the equation
$\Lmin(G)^k =\Lmin(G^{(k)})$.

\begin{example} {\rm ({\em The 5-cycle})
If $G$ is the 5-cycle $C_5$, then
$$G^{(3)}=K_5 + 2C_5.$$
Let $z=\Lmin(G)$.  Then $z^3=\Lmin(G^{(3)})$ and so, by Theorem
\ref{Dthm},
$$z^3 \ge 2z -1 \quad \mbox{or} \quad (z-1)(z^2+z-1) \ge 0.$$
Thus, $\Lmin(G)=z \ge -(1+\sqrt{5})/2\approx -1.618$. In Example \ref{srgeg},
we shall see that equality holds.
}\end{example}

\begin{example} \label{pete}{\rm ({\em The Petersen graph})
If $G$ is the Petersen graph, then
$$G^{(3)} = 3G +2K_{10}.$$
Let $z=\Lmin(G)$.  Then $z^3=\Lmin(G^{(3)})$ and so, by Theorem
\ref{Dthm},
$$z^3 \ge 3z -2 \quad \mbox{or} \quad (z-1)^2(z+2) \ge 0.$$
Thus, $\Lmin(G)=z \ge -2$.  Also we see that $\Lmin(G)=-2$,
by noticing that the 6-cycle $C_6$ is an induced subgraph of $G$.

Of course, the eigenvalues for the 5-cycle and for the
Petersen graph, indeed, for any strongly regular graph
$G(n,k,a,c)$,
may be found immediately from the equation on the adjacency matrix (see \cite[p.218]{GandR}):
\begin{equation}
A^2=kI +aA+c(J-I-A). \label{srg}
\end{equation}
For, by multiplying (\ref{srg}) by an eigenvector $\vx$ orthogonal to $\ones$
it follows that the only eigenvalues of $G$ other
than $k$ are the roots $\theta>0$ and $\tau<0$ of the quadratic equation
\begin{equation}
x^2-(a-c)x-(k-c)=0. \label{quadeqn}
\end{equation}
In particular, for the Petersen graph, $\theta=1$, $\tau=-2$.
}\end{example}

\begin{example} \label{srgeg}{\rm ({\em Strongly regular graphs })
For a strongly regular graph $G$, taking $H=G^{(3)}$ in Theorem
\ref{Dthm}
often leads to the exact value of $\Lmin(G)$.  To see this note that
multiplying (\ref{srg}) by $A$ and substituting (\ref{srg}) for $A^2$
gives an equation of the form
$$ A^3 =  rA + s(J-I)+tI$$
for some nonnegative integers $r,s,t$ depending on $n,k,a,c$.
Thus
$$G^{(3)} = rG + sK_n +tI_n.$$
By Theorem \ref{Dthm}, if $z=\Lmin(G)$, then
$$z^3 = \Lmin(G^{(3)})\ge rz - s + t.$$
Thus, $z$ is at least as large as the minimum root of the cubic
$x^3-rx+s-t$.  Two of the roots are $\theta$ and $\tau$, inherited from
(\ref{quadeqn}), and the other is necessarily $-(\theta + \tau)=c-a$
since the coefficient of $x^2$ is 0. Thus, $\Lmin(G)=z\ge
\min\{ {c-a, \tau}\}$. Therefore, taking $G^{(3)}$ in Theorem
\ref{Dthm}  gives $\Lmin(G)=\tau$ when $G$ is a strongly
regular graph such that $\tau\le c-a$. In particular,
$\Lmin(G)=\tau$ when $a \le
c$, a condition that must be satisfied by at least one of a
strongly regular graph and its complement.
}\end{example}

\begin{example}\label{dodec} {\rm ({\em The dodecahedral graph })
Let $G$ be the plane graph whose vertices and edges are those of
the dodecahedron.  Then $G$ is 3-regular and the 20 face 5-cycles
of $G$ constitute a decomposition $\D$ of $2G$ with precisely
3 cycles through each vertex.  Thus by Theorem
\ref{Dthm}, $\Lmin(G) \ge -3\Lmin(C_5)/2 = -3(1+\sqrt 5)/4 \approx -2.427$.
The exact value is $\Lmin(G) = -\sqrt 5 \approx -2.236$.
}\end{example}

If some weighted graph $H^j$ in a decomposition $\D$ is disconnected, it is clear
that replacing $H^j$ in $\D$ by its set of weighted components cannot weaken
the estimate in Theorem \ref{Dthm}.  Thus, there is no loss in restricting the weighted graphs
in a decomposition $\D$ to be connected.  In Example \ref{srgeg},
we could replace each loop graph $\I_n$ by its $n$ separate individual loops $\I_1=\J_1$.

By the {\em type} of a weighted graph, we mean its underlying unweighted graph.
Each choice of types for the weighted graphs in $\D$ in (\ref{Dineq}) leads to a lower bound on $\Lmin(G)$
by maximizing over all weighted graphs of that type and so, when applied to a simple graph $G$,
yields a new graph parameter. Of course, equality holds if all
types of weighted graph are allowed (take $\D$ to be $G$ itself).
The trick is to pick a family of graphs that are easy to deal with and that often yield good lower bounds in
(\ref{Dineq}) when maximized over all weightings.
The conditions for equality in Theorem \ref{Dthm} suggest
that the bound (\ref{Dineq}) might often
be best for decompositions $\D$ of a weighted graph $H(V,w)$ that employ weightings of connected graphs
whose minimum eigenvalues have large multiplicity and small absolute value.
Such are the simple complete graphs and looped complete graphs on subsets of $V$.
For when $n>1$, $K_n$ has least eigenvalue $-1$ with
multiplicity $n-1$, while $\J_n$
has least eigenvalue $0$ with multiplicity $n-1$.  (In both cases, the eigenvectors
associated with the minimum eigenvector are the nonzero vectors $\vx$ such that
$\sum_u x_u=0$.)   Because $K_1$ has no edges, it is never used in a decomposition.
The graph $\J_1$ will be called a {\em loop}. It may be used in a decomposition, noting carefully that
its least eigenvalue is $+1$.  We are therefore led to the following definition.

A {\em complete graph decomposition}  of a weighted graph $H(V,w)$ is a decomposition $\C=\{a_1C^1, \ldots a_mC^m\}$ of $H(V,w)$  consisting of scalar multiples of complete graphs, looped or simple. Because negative weights are allowed, cancellation may occur (as in Example \ref{ctemultipart}), so the graphs $C^j=K^j$ or $J^j$ in $\C$ need not be subgraphs of $H$. Taking $\D=\C$ in Theorem \ref{Dthm}, we obtain the following corollary.

\begin{corollary}\label{Ccorol}
Let $\C=\{a_1C^1, \ldots, a_mC^m\}$ be a complete graph decomposition of a weighted graph $H=H(V,w)$ and,
for each vertex $u\in V$, let $\Lmin(\C_u)$ equal the sum of the minimum eigenvalues of the
complete graphs in $\C$ that contain the vertex $u$.
Then
\begin{eqnarray}\label{Cineq}
\Lmin(H)\geq \min_u\Lmin(\C_u).
\end{eqnarray}
Let $\Lmin(\C)=\min_u\Lmin(\C_u)$.
Then equality holds in (\ref{Cineq})
if and only if there is a vector $\vx\ne\zeros$ of real numbers assigned to the vertices of $H$ such that
\begin{enumerate}\label{eqcond}
\item $x_u=0$ whenever $\Lmin(\C_u)>\Lmin(\C)$;
\item $\vx$ is constant on each vertex set $V_j$ for which $a_j<0$ and,
$\sum_{u\in V_j} x_u = 0$ for each vertex set $V_j$ of order greater than 1 for which $a_j>0$.
\end{enumerate}
\end{corollary}

\begin{example} \label{ctemultipart}  {\rm ({\em Complete multipartite graphs })
Let $G=K_{n_1,\ldots,n_m}$ be the complete multipartite graph with vertex parts $V_1, V_2, \ldots, V_m$ of orders
$n_1\ge n_2 \ge \cdots \ge n_m$.
Then
$$G=J_n -J_{n_1} - \cdots -J_{n_m}$$
where $\J_n$ is the looped complete graph of order $n=n_1+ \cdots + n_m$
and, for $1\leq j\leq m$,  $-\J_{n_j}$ is the negatively weighted looped complete graph on $V_j$.
Since $\Lmin(\J_n)=0$ and $\Lmin(-\J_{n_j})=-n_j$, we have
$\Lmin(G)\ge -n_1$ by Corollary \ref{Ccorol}.  Moreover, it is straightforward to check
that conditions 1 and 2 imply that equality holds if and only if $n_1=n_2$.
Of course, using the characteristic polynomial of $G$ \cite[p.74]{CDS}, it follows that $G$ has precisely $m-1$ negative
eigenvalues and that they interlace $-n_1, -n_2, \ldots, -n_m$.   Also, $G$ has one positive eigenvalue (in fact, this characterizes
the complete multipartite graphs \cite[p.163]{CDS}).  All
remaining eigenvalues are equal to 0.
}\end{example}

If the values of the weight function $w$ of a weighted graph $H(w)$ are nonnegative integers,
then $H(w)$ may be regarded as a {\em multigraph}
with $w(uv)$ {\em distinct unweighted} edges between each unordered
pair $uv$ of vertices of $H$.
We call $w(uv)$ the {\em multiplicity} of $uv$.
To emphasize this distinction,
we use the notation $\hat G=\hat G(w)=\hat G(V,w)$ for multigraphs
and continue to use $G$ for graphs (simple or looped) and $H$ for weighted graphs.

\begin{example} \label{linegrapheg}  {\rm ({\em Line graphs of multigraphs})
Let $\hat G=\hat G(w)$ be a  multigraph with maximum edge multiplicity $\mu=\max_{u\sim v}w(uv)$.
The {\em line graph} $L(\hat G)$ of $\hat G$ has the edges of $\hat G$ as vertices.
Two edge vertices  of $L(\hat G)$ are
adjacent if they have precisely one common end vertex in $\hat G$.  Thus edges in $\hat G$ with the same two endpoints
are nonadjacent as vertices in $L(\hat G)$.  Note that $L(\hat G)$ is a simple graph.  Also, if a loop at $u$ in $\hat G$ is replaced
by an edge with one end at $u$ and the other at an additional new vertex, then the line graph
is not changed.  Thus we may assume that $\hat G$ has no loops.

The line graph $L(\hat G)$ has a natural decomposition into complete multipartite graphs.  To see this,
for each vertex $u$ of $\hat G$, let $T(u)$ be the subgraph of $L(\hat G)$ induced by the {\em claw} at $u$,
that is by the edges incident to $u$ in $\hat G$.
The subgraph $T(u)$ of $L(G)$ is a complete multipartite graph and, because $\hat G$ has no loops,
the part sizes of $T(u)$ are equal to the multiplicities of the edges incident to $u$ in $\hat G$. Thus, by Example \ref{ctemultipart},
$\Lmin(T(u)) \ge -\mu$.
Because adjacent edge vertices of $L(\hat G)$ have precisely one vertex in common in $\hat G$, it follows that the graphs $T(u)$
decompose $L(\hat G)$.  Also since each edge of $\hat G$ has two distinct endpoints, each edge vertex of $L(\hat G)$ is in precisely two
graphs of the decomposition.  Thus, by Theorem \ref{Dthm}, if $\hat G=\hat G(w)$ is a loopless multigraph with maximum edge multiplicity $\mu$, then
\begin{equation}\label{linegraph}
\Lmin(L(\hat G)) \ge \min_{u\sim v} \Lmin(T(u))+\Lmin(T(v)) \ge -2\mu.
\end{equation}

Equality can be attained in \eqref{linegraph}.  To see this, note first that if $G$ is a simple graph and $M$ is the adjacency matrix
of $L(G)$, then the adjacency matrix of $L(\mu G)$ is the Hadamard product $M\otimes J_{\mu}$.  Thus, $\Lmin(L(\mu G))= \mu\Lmin(L(G)).$ Therefore, if $\hat G$ has maximum edge multiplicity $\mu$ and $\mu G$ is an induced subgraph of $\hat G$ with $\Lmin(L(G))=-2$
(see Example \ref{twig}), then $\Lmin(L(\hat G)) \le \Lmin(L(\mu G)) = -2\mu$ and so equality is attained in (\ref{linegraph}). We leave it as an problem to figure out if it is possible to characterize the multigraphs for which equality \eqref{linegraph} is attained.
}\end{example}

\begin{example} \label{twig} {\rm ({\em Twig replication and generalized line graphs })
There are interesting cases where the lower bound (\ref{linegraph}) can be improved.
Suppose that a loopless multigraph $\hat G=\hat G(w)$ is formed from a connected simple graph $G$
by optionally increasing the multiplicity of {\em twigs} of $G$, that is, of edges of $G$
(if any) that have an end vertex of degree 1.  Decompose $L(\hat G)$ by the complete multipartite graphs
$T(u), u\in V(\hat G)$ as in Example \ref{linegrapheg}.
If $u$ is a vertex of degree 1 in $G$, then $T(u)$ will have no edges in $L(\hat G)$ and may be omitted.
Now further decompose each subgraph $T(u)$ by graphs $\J$ and $-J$ of appropriate orders as in Example \ref{ctemultipart}.
Because of the construction,
the vertex parts of size 2 or more that occur in the graphs $T(u)$ will be vertex disjoint, and so, each
vertex of $L(\hat G)$ will be in at most one $-\J$ graph of order 2 or more.  Also, each edge vertex of $L(\hat G)$ will be in at most
two $-J_1$ graphs in the decomposition since it is in at most two $T(u)$'s. Thus, by Theorem \ref{Dthm}, if $\hat G=\hat G(w)$
is a multigraph formed from a simple graph $G$ by replicating twigs, then
\begin{equation}\label{genzdlinegraph}
\Lmin(L(\hat G)) \ge \min\{-2, -\mu\},
\end{equation}
where $\mu=\max_{uv}w(uv)$ is the maximum twig multiplicity in $\hat G(w)$.

When $\mu=1$, we have $G=\hat G$ and $\Lmin(L(G)) \ge -2$ where $L(G)$ is the
usual line graph of a simple graph $G$. In this case, each vertex $uv$ of $L(G)$ is in precisely two complete graphs, $T(u), T(v)$,
and the conditions for equality can be shown to imply a result of Doob (see, for example, \cite[p.29]{CRS}) which states that
$\Lmin(G)>-2$ if and only if each component of $G$ is either a tree or is odd-unicyclic.

When $\mu=2$, we again have $\Lmin(L(\hat G)) \ge -2$.  The graphs $L(\hat G)$ with $\mu=2$  are
the {\em generalized line graphs} of Hoffman \cite{Hoff} (see also \cite{CDS,CRS2} or \cite[p.6]{CRS}).  The usual proofs that $\Lmin(L(\hat G))\ge -2$ for
a generalized line graph $L(\hat G)$ employ modifications of the vertex-edge incidence matrix of $\hat G$ \cite{Hoff}, \cite[p.6]{CRS}.
}\end{example}

Let $\LminC(H)$ be the best possible estimate of $\Lmin(H)$ that can be obtained
in (\ref{Cineq}); that is, let
$$\LminC(H) =\sup_\C\Lmin(\C),$$
where the supremum is taken over all complete graph decompositions $\C$
of $H=H(V,w)$.  To see that the supremum is attained,
we show that $\LminC(H)$ is the optimal value of a linear programming problem.

Let  $M$ be the
incidence matrix with rows indexed by all of the (unordered) vertex pairs $uv$
and columns indexed by all complete graphs, looped and simple, with vertex sets contained in $V$.
(Note that because cancellation may occur,
all complete graphs must be taken, whether or not they are subgraphs of $H(V,w)$.)
Let $\vw$ be the weight vector determined by the given weight function $w$ on $G$;
that is, $\vw_{uv}=w(uv)$ for each unordered vertex pair $uv$.
Then a vector $\vz$ indexed by the complete graphs specifies a complete graph decomposition $\C$ of $H(V,w)$ with weights $\vz$
if and only if $M\vz=\vw$.  Now let $N$ be incidence matrix with rows indexed by the vertices and columns by the complete graphs and let $L$
be the diagonal matrix with diagonal entries equal to the minimum eigenvalues of
all of the complete graphs with vertex sets contained in $V$. Then $\Lmin(\C)$ is the smallest number $\L$
such that $NL\vz \ge \L \ones$.  Thus $\LminC(H(V,w))$ is the optimal value of the following
linear programming problem in the variables $\vz$, $\L$.
\begin{eqnarray}
\mbox{Minimize} \  & \L & \notag \\
\mbox{  Subject to} &   & \label{LP}\\
M\vz & = & \vw \notag\\
NL\vz & \ge & \L\ones \notag.
\end{eqnarray}
Thus, $\LminC(H(V,w))$ is  attained for some complete graph decomposition
$\C$ of $H(V,w)$.   Moreover, if $w$ is rational valued, then the optimal value
$\L^*$ is rational and an optimal vector $\vz^*$ giving equality in (\ref{LP}) may be chosen
to have rational entries.  Consequently,
there is a positive integer $\mu$ such that $\mu\L^*$ is an integer and $\mu\vz^*$
has integer entries.  When $w$ is rational valued,
this observation allows us to work with decompositions $\D$ consisting of integer multiples of complete graphs, looped or simple,
as long as multiples $\mu H$ of the weighted graph $H$ are employed.

Thus, for a simple graph $G$,  the graph parameter $\LminC(G)$ has the following equivalent definition:
$$\LminC(G)=\min_{\mu,\,\C} \Lmin(\C)/\mu$$
where the minimum is taken over all positive integers  $\mu$ and all decompositions $\C$
of $\mu G$ by integer multiples (positive or negative) of complete graphs $C^j=K^j$ or $J^j$.

It is perhaps impossible to classify the simple graphs $G$ for which the parameter $\LminC(G)$
equals the least eigenvalue $\Lmin(G),$ but there are a few simple observations that limit the
graphs for which equality holds.
Because the characteristic polynomial of a graph $G$ (or multigraph  $\hat G$) is monic with integer coefficients, every
rational root is an integer. Therefore, if the rational number $\LminC(G)$  is
not an integer, then $\Lmin(G)> \LminC(G)$.  Also, if $\Lmin(G)$ happens to be irrational
(as for example, for the 5-cycle), then $\Lmin(G)>\LminC(G)$.

\section{Smallest eigenvalue and clique partitions}\label{sec:clique}

In this section, we further restrict the type of decompositions $\C$ in (\ref{Ccorol}) to a special type that
often appear in the literature, {\em clique partitions}.

A {\em clique} in a multigraph $\hat G$ is a simple complete subgraph.
A clique partition of a (necessarily loopless) multigraph $\hat G$ is a collection $\K=\{ K^1, \ldots, K^m \}$
of cliques of $\hat G$ whose edge-sets
partition the edge-set of $\hat G$.  Here we do not weight the cliques, but may take the same clique more than once.
Consequently, because all of the cliques in a clique partition have least eigenvalue $-1$,
for each vertex $u$ of $\hat G$,  we have the convenient expressions
\begin{equation} \label{rdefn}
\Lmin(\K_u)=-r_u(\K) \mbox{ and } \Lmin(\K)=-r(\K)
\end{equation}
where $r_u=r_u(\K)$ is the number of cliques in $\K$ that contain the vertex $u$ and $r(\K)=\max_u r_u$.  Thus,
\begin{equation}
\Lmin(\hat G)\ge -r(\K)
\end{equation}
We are mainly interested in graphs $G$ that are simple.
Because we are now only allowing copies of cliques in our partitions,  taking scalar multiples
$\hat G=\mu G$ can sometimes improve our bound
on $\Lmin(G)$.  We have the following corollary to Theorem \ref{Dthm}.

\begin{corollary}\label{Kcorol}
Let $\K$ be a clique partition of a multiple $\mu G$ of a simple graph $G$.
Then
\begin{eqnarray}\label{Kineq}
\Lmin(G)\geq -\frac{r(\K)}{\mu}
\end{eqnarray}
with equality
if and only if there is a vector $\vx\ne\zeros$ of real numbers assigned to the vertices of $G$ such that
\begin{eqnarray}\label{Kcond}
x_u=0 \mbox{ whenever } r_u(\K)<r(\K); \mbox{ and,}
\sum_{u\in K} x_u = 0 \mbox{ for each clique } K\in\K.
\end{eqnarray}
\end{corollary}

The conditions (\ref{Kcond}) for equality in Corollary \ref{Kcorol} may be restated in a convenient matrix form.
If $\K=\{ K^1, \ldots, K^m \}$ is a clique partition of a multiple $\mu G$ of a simple graph $G$, let
$N=N(\K)$ be the $n\x m$ vertex-clique incidence matrix of $\K$ with rows indexed by the vertices of $G$ and columns by the
cliques in $\K$.  Thus, the $(u,K^j)$-entry of $N$ is $1$ if $u\in K^j$ and is zero otherwise.  Then equality holds in (\ref{Kineq})
if and only if there is a vector $\vx\neq\zeros$ indexed by the vertices of $G$ such that $N^T\vx=\zeros$ and $x_u=0$ whenever
$r_u(\K) < r(\K)$.

\begin{example}{\rm ({\em Line graphs}) It is an immediate consequence of Corollary \ref{Kcorol} that
if a simple graph $G$ can be edge-partitioned by simple cliques so that each vertex is in at most two
of the cliques, then $\Lmin(G)\ge -2$.  But we have already encountered these graphs in Example \ref{linegrapheg}:
by a result of J. Krausz
\cite{Kra}, they are precisely the line graphs of simple graphs.
}\end{example}

\begin{example}{\rm ({\em Partial geometries})
A partial geometry $pg(K,R,T)$ is an incidence structure of points and lines such that any two points are incident with at most one line,
every line has $K$ points, every point is on $R$ lines, and for any line $L$ and any point $p\notin L$, there are exactly $T$ lines through $p$ that intersect $L$. Partial geometries were introduced by Bose \cite{Bose} along with strongly regular graphs. The point graph of a partial geometry $pg(K,R,T)$ is the graph whose vertices are the points of the geometry where two vertices/points are adjacent if there is a line that contains them. It is known (see \cite[Problem 21H]{vLW}) that the point graph of a partial geometry $pg(K,R,T)$ is a strongly regular graph with smallest eigenvalue $-R$. Note that the edge set of this graph can be partitioned into cliques (corresponding to the lines of the geometry) such that each vertex is contained in exactly $R$ cliques. Corollary \ref{Kcorol} with $\mu=1$ implies that the smallest eigenvalue of this graph is at least $-R$ which is tight. The point graphs of partial geometries also appear in \cite{CX} where it is proved that certain random walks on them mix faster than the non-backtracking walks considered in \cite{ABLS}.
}
\end{example}

\begin{example}{\rm  ({\em The Johnson graphs}) Let $v,k$ be positive integers with
$v\ge 2k$.   The Johnson graph $J(v,k)$ has
the $k$-subsets of a $v$-set $X$ as vertices with $S,T\subset X$
adjacent if $|S\cap T|=k-1$.   If $C$ is a $(k-1)$-subset of $X$, then the
set $K(C)$ of all $k$-subsets of $X$ that contain $C$ is the vertex set of a clique in $J(v,k)$.
Each pair $S,T$ of adjacent vertices is in precisely
one such clique, the clique $K(S\cap T)$.  Thus, the family $\K=\{ K(C) : |C|=k-1, C\subset X\}$
is a clique partition of $J(v,k)$.
Also, $r_S(\K)=k$ for each $S\in V$. By Corollary \ref{Kcorol}, $\Lmin(J(V,k))\ge -k$.  Moreover, a nonzero vector $\vx=(x_S)_S$ satisfies the conditions for equality if and only if $\sum_{S\in K(C)} x_S =0$ for each clique $K(C)$, $|C|=k-1$.
This is a system of $\binom{v}{k-1}$ homogeneous linear equations in $\binom{v}{k}$ variables and so has
a nontrivial solution. Thus $\Lmin(J(v,k))=-k$ (with multiplicity $\binom{v}{k}-\binom{v}{k-1}$, since the
constraints can be shown to be linearly independent).  There are explicit formulas for all of the eigenvalues
and multiplicities of the relation graphs of the Johnson schemes and, in particular, for the Johnson graphs \cite[p.413]{vLW}.
}\end{example}

Corollary \ref{Kcorol}  leads us to a graph parameter based on clique partitions.
For a simple graph $G$, let
$\LminK(G)$ be the best possible estimate of $\Lmin(G)$
that can be obtained
using clique partitions of scalar multiples of $G$; that is, let
\begin{equation}\label{rstardefn}
\LminK(G)= -\max_{\mu,\, \K} r(\K)/\mu
\end{equation}
where the maximum is taken over all positive integers $\mu$ and all clique partitions $\K$
of $\hat G=\mu G$.   Then
\begin{equation} \label{rstarineq}
\Lmin(G)\ge \LminC(G) \ge \LminK(G).
\end{equation}
As in (\ref{LP}) with $\LminC(G)$, a linear programming problem shows that $\LminK(G)$ is attained by some $\mu, \K$ and is rational.
\begin{remark}\label{bipartite} {\rm
As with the equality $\Lmin(G)=\LminC(G)$, it may be impossible to classify the simple graphs $G$ for which $\Lmin(G)= \LminK(G)$,
but there are conditions that restrict the possible simple graphs. Again, because rational roots of monic polynomials are integers,
$\LminK(G)$ must be an integer if $\Lmin(G)=\LminK(G)$. We also note that we may as well restrict our attention
to simple graphs $G$ that contain triangles, $K_3$.  For if $G$ is $K_3$ free, then the only cliques in $G$
are edges and it follows that $\LminK(G)=-\Delta(G)$, the maximum vertex degree in $G$.
Thus, if $G$ is connected and triangle free,  $\Lmin(G)=\LminK(G)$ if and only if $G$ is a regular bipartite graph.
(This can be seen by standard results, or from Remark \ref{Kedges} below.)
}\end{remark}

\medskip
Another limitation on the equality $\Lmin(G)=\LminK(G)$  follows by noting that
conditions (\ref{Kcond}) for equality in Corollary \ref{Kcorol} can sometimes be extended
if $r_u(\K)<r(\K)$ for some vertex $u$.
Let $V^1=\{ u\in V  :  r_u(\K)=r(\K) \}$
and let $\vx$ be a vector satisfying conditions (\ref{Kcond}).
If $V^1=V$, stop. If $V^1\neq V$, then $x_u=0$ for all $u\in V\backslash V^1$. There may now be a clique $K\in \K$ that
meets $V^1$ in only one vertex $v$, say. Then $x_v=0$ since $x_u=0$ for all $v\in V(K)\backslash \{v\}$
and $\sum_{u\in V(K)} x_u =0$.  Let $V^2$ be the set of vertices obtained by deleting all vertices $v\in V^1$
for which there is a clique $K\in \K$ that meets $V^1$ only in $v$.  Then $x_u=0$ for all $u\in V\backslash V^2$.
Repeat this last step.  That is, given $V^i$, let
$$V^{i+1}=V^i\backslash \{v  :  V(K)\cap V^i =\{ v \} \mbox{ for some } K\in \K\}.$$
Eventually, we obtain a set $V^*$ (possibly empty) such that each clique in $\K$ is either disjoint from $V^*$
or else meets $V^*$ in two or more vertices.  We call the vertices in $V^*$ the $\K$-{\em essential vertices}.
Note that if $\vx$ satisfies conditions (\ref{Kcond}), then $x_u=0$ for all $u\in V\backslash V^*$.
Also, because each clique in $\K$ is either disjoint from $V^*$
or meets $V^*$ in two or more vertices, if $V^*\neq \emptyset$, we must have $r_u(\K^*)=r_u(\K)=r(\K)$ for each vertex $u\in V^*$.
Thus,
\begin{equation}\label{starequals}
r_u(\K^*)=r(K^*) \mbox{ for all } u\in V^*, \mbox{ and so } r(\K^*)=r(\K).
\end{equation}
We now have the following result.

\begin{lemm}\label{essential}  Let $G$ be a simple graph with vertex set $V$.
Let $\K$ be a clique partition of $\mu G$ and let $V^*$ be the set of $\K$-essential vertices.
Then $$\Lmin(G)=-r(\K)/\mu \mbox{ if and only if } V^*\neq\emptyset \mbox{ and } \Lmin(G^*)=-r(K^*)/\mu$$
where $G^*=G[V^*]$ and $K^*$ is the set of nonempty restrictions of cliques in $\K$ to $V^*$.
Moreover, if $\Lmin(G)=-r(\K)/\mu$, then $\Lmin(G)=\Lmin(G^*)$ and $r_u(\K^*)=r(\K)$ for all $u\in V^*$.
\end{lemm}


\begin{proof}
Suppose that $\Lmin(G)=-r(\K)/\mu$. Then there is a vector $\vx\ne \zeros$ satisfying conditions (\ref{Kcond}) for equality in
Corollary \ref{Kcorol}.  Thus $V^*$ is nonempty, otherwise the observations above
imply that $x_u=0$ for all $u\in V\backslash V^*=V$, a contradiction.
Let $\vx^*$ be the restriction of $\vx$ to $V^*$. Because $\vx\ne\zeros$ and $x_u=0$ for $u\in V\backslash V^*$, we have $\vx^*\ne \zeros$
and $\sum_{u\in V(K^*)} x^*_u = 0$ for each restricted clique $K^*=K[V^*]$. Also, by \ref{starequals},
$r_u(\K^*)=r(\K^*)=r(\K)$ for $u\in V^*$.
Thus,  $\vx^*$ is a nonzero vector that satisfies the conditions (\ref{Kcond}) for equality for the clique partition $\K^*$ of $\mu G^*$.
Therefore, $\Lmin(G^*)=-r(\K^*)/\mu$. Also, $\Lmin(G)=\Lmin(G^*)$ since $r(\K^*)=r(\K)$.

Suppose now that $\Lmin(G^*)=-r(\K^*)/\mu$. By Corollary \ref{Kcorol} and (\ref{starequals}),  $\Lmin(G)\ge -r(\K)/\mu=-r(\K^*)/\mu.$
Thus, $\Lmin(G) \ge \Lmin(G^*)$.  The reverse inequality holds since $G^*$ is an induced subgraph of $G$.  Thus, $\Lmin(G) = \Lmin(G^*)$.
\end{proof}


\begin{remark}\label{Kedges}{\rm
For Lemma \ref{essential} to hold, it is necessary that the clique partition $\K^*$ of $\mu G^*$ be obtained by the restrictions of the cliques in $\K$.  In particular, if there is a clique partition $\tilde \K$ of $\mu G^*$ with $r(\tilde \K) < r(\K)$, then $\Lmin(G^*)=-r(K^*)/\mu$ and so $\Lmin(G)>-r(\K)/\mu$.

In the special case that each clique in $\K^*$ is a single edge (in particular, if $G^*$ is bipartite), it follows from Lemma
\ref{essential} and the conditions (\ref{Kcond}) that $\Lmin(G)=-r(\K)/\mu$ if and only if
$G^*$ is regular and some component is bipartite.  The key observation needed here is that if $x^*_u + x^*_v=0$
for each edge $uv$ in $G^*$, then the set of vertices $\{ u\in V^* : x_u\ne 0 \}$ is the vertex set of a union of
connected components of $G^*$ and the two subsets $\{ u\in V^* : x_u > 0 \}$, $\{ u\in V^* : x_u< 0 \}$ are a bipartition.
}\end{remark}

\begin{example}{\rm
Let $G^0$ be a simple $k$-regular bipartite graph and let $U$ be a set of vertices disjoint from $V^0=V(G^0)$.
Let $G$ be a simple graph obtained from $G^0$ by replacing some (or all) of the edges $uv$
of $G^0$ by cliques in $U\cup V^0$ that meet $G^0$ in the vertices $u,v$ only so that:
\begin{enumerate}
\item  Each vertex of $U$ is in at most $k-1$ of the cliques.
\item  Each pair of distinct vertices in $U$ is in at most one clique.
\end{enumerate}

\noindent
Let $\K$ be the cliques that replaced the edges together with the edges of $G^0$ that
were not replaced.  Then $G^*=G^0$, $K^*$ is the edge set of $G^*$ and $\Lmin(G)= -r(\K) = -k$ since $\Lmin(G^*)=-r(\K^*) = -k$.
}\end{example}

In Lemma \ref{essential}, we observed that if the set $V^*$ of $\K$-essential vertices is nonempty
and $\K^*$ is the set of nonempty restrictions
of cliques in $\K$ to $V^*$, then $r_u(\K^*)=r(\K)$ for each
vertex $u\in V^*$.
Therefore, when searching for simple graphs $G$ for which $\Lmin(G)=\LminK(G)$,
we may focus our attention on simple graphs $G^*$ such that some multiple $\mu G^*$
has a clique partition $\K^*$ for which
$r_u(\K^*)$ is constant and $\Lmin(G^*)=-r(\K^*)/\mu$.

\begin{lemm}\label{deltbnd}
Let $G$ be a simple graph with maximum vertex degree $\Delta$.
If $\K$ is a clique partition of $\mu G$
and $c$ is the smallest of the orders of the cliques in $\K$,
then
\begin{equation*}
\frac{r(\K)}{\mu}\leq \frac{\Delta}{c-1},
\end{equation*}
where equality holds if and only if
for some vertex $u$ of maximum degree in $G$,
each clique in $\K$ containing $u$ has order $c$.
\end{lemm}

\begin{proof} Let $d_u$ denote the degree of $u$ in $G$.
There are $r_u(\K)$ cliques of $\K$ containing vertex $u$.
Since each of these cliques cover at least $c-1$ of the
$\mu d_u$ edges incident to $u$ in $\mu G$,
$r_u(\K) (c-1) \le \mu d_u \le \Delta \mu.$  Thus
$$\frac{r(\K)}{\mu}\le \frac{\Delta}{c-1}$$
with equality if and only if the stated condition holds.
\end{proof}

The estimate on $r_u(\K)$ in the proof of Lemma \ref{deltbnd} can be improved
if the number of cliques of smallest order at $u$ is known.

\begin{lemm}\label{smallcliques}
Let $G$ be a simple graph and let $\K$ be a clique partition of $\mu G$. Let $c$ be the
smallest of the orders of the cliques in $\K$. For each vertex $u$,
denote by $d_u$ its degree in $G$,  and by $e_u$ the number of
cliques of order $c$ in $\K$ that contain $u$. Then
\begin{equation}
r_u(\K)\leq \frac{\mu d_u + e_u}{c}
\end{equation}
with equality if and only if each clique in $\K_u$ has order $c$
or $c+1$.
\end{lemm}
\begin{proof}
Of the $\mu d_u$ edges incident to $u$ in $\mu G$,
the $e_u$ cliques of order $c$ cover $e_u(c-1)$ of the edges,
while the remaining $r_u-e_u$ cliques cover at least $(r_u-e_u)c$
of the edges.  Thus $\mu d_u \ge e_u(c-1)+(r_u-e_u)c$ or
$r_u(\K)\leq (\mu d_u + e_u)/{c}$ with equality if and only if
each clique in $\K_u$ of order greater than $c$ has order $c+1$.
\end{proof}

The {\em direct product} (or simply, {\em product}) of two graphs $G^1$ and $G^2$, denoted $G^1\x G^2$, has
vertex set $V(G^1)\x V(G^2)$ with vertices $(u,v)$  and $(u',v')$
adjacent if $u$ and $u'$  are adjacent in $G^1$ and $v$ and
$v'$ are adjacent in $G^2$.  If either $G^1$ or $G^2$ is simple, then $G^1\x G^2$ is simple.
The adjacency matrix of $G^1\x G^2$ is
$A(G^1)\otimes A(G^2)$, and the eigenvalues are $\lambda_i(G^1)\lambda_j(G^2)$,
where $1\leq i \leq |V(G^1)|$ and $1\leq j\leq |V(G^2)|$.

The next example illustrates how taking clique partitions of a multiple of a simple graph $G$
can sometimes give a better bound on $\Lmin(G)$ than clique partitions of $G$ alone.

\begin{example}\label{product} {\rm ({\em Direct products of simple complete graphs})
Let $\K$ be the set of all simple cliques of order $m$ in $K_m\x
K_n, m\leq n$.  Then $\K$ is a clique partition of
$\mu(K_m\x K_n)$ where $\mu$ is the number of simple cliques of $\K$ containing an edge of $K_m\x K_n$.  Thus $\Lmin(K_m\x
K_n)\geq-\frac{\Delta}{c-1}=-\frac{(m-1)(n-1)}{n-1}=-(n-1)$.  But
$\Lmin(K_m\x K_n)=\min\{-(m-1),-(n-1)\}=-(n-1)$, so equality is
attained in Lemma \ref{deltbnd}. That equality is attained can also be seen later in
example \ref{composition}.

It was necessary to take a multiple of $K_m\x K_n$ in this example in order
to be able to have an edge-partition by cliques of order $m$. For,
it can be shown that results of Pullman et al. \cite{PSW}
imply that $K_m\x K_n$ can be edge-partitioned by cliques of order $m$ if
and only if there exists an $m\x n^2$ orthogonal array with $n$
constant columns.
}\end{example}

The result in Example \ref{product} can be extended to obtain $\Lmin(G^1\x G^2)$ in some cases.
Let $G^i$ be a $k_i$-regular graph for $i\in \{1,2\}$.  Suppose that
$\Lmin(G^i)=\LminK(G^i)$.  Furthermore, let $\K_i$ be
a clique partition of $\mu_i$ into cliques of order $c_i$ such that
$\LminK(G^i)=-r(\K_i)/\mu_i$. Then
$$\Lmin(G^i)=\LminK(G^i)=-\frac{k_i}{c_i-1}.$$
Let $G=G^1\x G^2$.  The set of all subgraphs $\kappa_1\x
\kappa_2$ of $G$, where $\kappa_i\in \K_i$, partitions the edge
set of $\mu_1\mu_2 G$.  Each of these subgraphs is isomorphic to
$K_{c_1}\x K_{c_2}$.  Suppose $c_1\leq c_2$.  If $\mu$ is the number
of $c_1$-cliques containing an edge in $\kappa_1\x \kappa_2$, then
the set of all such cliques, over all subgraphs $\kappa_1\x\kappa_2$,
partitions the edge set of $\mu\mu_1\mu_2 G$. Therefore
$$\Lmin(G)\geq \frac{-\Delta(G)}{c_1-1}=-\frac{k_1k_2}{c_1-1}$$
where the inequality follows from Lemma \ref{deltbnd}.  But
\begin{eqnarray*}
\Lmin(G) & = &
\min\{-\Lmin(G^1)k_2,-k_1\Lmin(G^2)\}\\
       & = & -\max\left\{\frac{k_1k_2}{c_1-1},\frac{k_1k_2}{c_2-1}\right\}\\
       & = & -\frac{k_1k_2}{c_1-1}
\end{eqnarray*}
so equality is attained.  Thus $\Lmin(G)=\LminK(G)$.

Therefore we have the following result.
\begin{theor}
Let $G^i$ be a $k_i$-regular graph for $i\in \{1,2\}$.  Suppose that
$\Lmin(G^i)=\LminK(G^i)$ and there is a clique partition $\K_i$ of
$\mu_i$ into cliques of order $c_i$ such that $\LminK(G^i)=-r(\K_i)/\mu_i$.
Then $\Lmin(G^1\x G^2)=\LminK(G^1\x G^2)$.
\end{theor}

An {\em independent set} of vertices in a simple graph $G$ is a set of vertices no two of
which are adjacent.  The independence number, $\alpha=\alpha(G)$, is the maximum cardinality of
an independent set of vertices in $G$.  Hoffman's ratio bound (see, for example, \cite[Lemma 9.6.2]{GandR})
asserts that if $G$ is a $k$-regular simple graph of order $n$, then
$\alpha(G) \le -n\Lmin(G)/(k-\Lmin(G))$. Thus
\begin{equation} \label{Hoffman}
\Lmin(G) \le -\frac{\alpha k}{n-\alpha},
\end{equation}
which is an upper bound on $\Lmin(G)$ for a $k$-regular simple graph of order $n$.

Let $B$ be the matrix whose rows are indexed by the
vertices of $G$ and whose columns are indexed by the independent sets of
$G$ with $B_{ij}=1$ if vertex $i$ is in independent set $j$.  The {\em
fractional chromatic number} of a simple graph $G$ is the minimum value of
$\ones^T \vx$ over all non-negative vectors $\vx$ that satisfy $B\vx\geq \ones$
(restricting $\vx$ to integral vectors yields the usual chromatic number).
For every simple graph $G$ of order $n$, we have $\chi(G)\ge \chi_f(G)\ge n/\alpha(G)$, where
$\chi(G)$ is the chromatic number of $G$ \cite[p. 136]{GandR}.
In particular, if $G$ is $k$-regular this implies the following weakened forms of
(\ref{Hoffman}):
\begin{equation} \label{fracbound}
\Lmin(G) \le -\frac{\alpha k}{n-\alpha} \le -\frac{k}{\chi_f(G)-1} \le -\frac{k}{\chi(G)-1}.
\end{equation}
If $G$ is a simple graph of order $n$, not necessarily regular, results of Lov\' asz \cite[Theorems 6,10 ]{Lovasz} imply that
$\chi(G)\ge \chi_f(G) \ge 1 -\Lmax(G)/\Lmin(G)$.
(The last bound is Hoffman's lower bound on $\chi(G)$.) This implies the following refinement of (\ref{fracbound}):
\begin{equation} \label{Lovasz}
\Lmin(G) \le -\frac{\Lmax(G)}{\chi_f(G)-1} \le -\frac{\Lmax(G)}{\chi(G)-1}.
\end{equation}

The following theorem gives sufficient conditions for equality to be attained in (\ref{Kineq}).
The {\em clique number} $\omega(G)$ is the order of the largest clique in $G$.
In general, $\omega(G)\leq \chi_f(G)$ (see \cite[Ch. 7]{GandR}).

\begin{theor}\label{chifomeg}
Let $G$ be a $k$-regular simple graph and suppose that for some $\mu$ there is a
clique partition $\K$ of $\mu G$ into cliques of order $\omega(G)$.
Then
$$\LminK(G)=-\frac{r(\K)}{\mu}=-\frac{k}{\omega(G) -1},$$
and so
$$\Lmin(G)\ge -\frac{k}{\omega(G) -1}.$$
For equality to hold, it is sufficient that $\omega(G)=\chi_f(G)$.
\end{theor}
\begin{proof}
By the definition of the parameter $\LminK(G)$, there is a positive integer $\hat\mu$ and a clique partition
$\hat K$ of $\hat\mu G$ such that
$$\LminK(G)=-\frac{r(\hat K)}{\mu} = -\max_u \frac{r_u(\hat K)}{\hat\mu}.$$
Since the $k\hat\mu$ edges of $\hat\mu G$ incident to vertex $u$ are partitioned by $r_u(\hat K)$ cliques of order at most
$\omega(G)$, we have
$r_u(\hat K)(\omega(G)-1) \ge \hat\mu k$.
Thus, $r_u(\hat K)/\hat\mu \ge k/(\omega(G)-1)$ and so $\LminK(G) \le -k/(\omega(G)-1)$.
But, by Lemma \ref{deltbnd} the clique partition $\K$ of $\mu G$ gives $\LminK(G)\ge -r(\K)/\mu \ge k/(\omega(G)-1)$,
so equality holds for $\LminK(G)$.

When $\omega(G)=\chi_f(G)$, equality holds for $\Lmin(G)$ by (\ref{fracbound}).
We also have the following direct proof using the conditions for equality in Corollary \ref{Kcorol}.
By \cite[Theorem 7.4.5]{GandR} there is some $n$ and some $k$ such that
$\chi_f(G)=n/k$ and the vertices of $G$ can be coloured by $k$-subsets of
$[n]$ such that subsets of adjacent vertices do not intersect.  If $c$ is a
clique of order $\chi_f$, then the $k$-subsets of the vertices of $c$
partition $[n]$. Associate to each element $i\in [n]$ a variable $a_i$.
Pick the $a_i$ so that $\sum_i a_i=0$.  Let $S_j$ be the subset of vertex
$j$ and $x_j=\sum_{i\in S_j}a_i$. Then the sum of the weights of the vertices of a clique is $\sum_i a_i$.
But we have chosen the $a_i$ so that this sum is equal to 0.  Thus if
$\K$ is a partition of the edge set into cliques of order
$\chi_f$, then $\vx$ is a nonzero vector satisfying the conditions for equality in Corollary \ref{Kcorol}.
Therefore by Corollary \ref{Kcorol}, $\Lmin(G)=-r(\K)/\mu=\LminK(G)$.
\end{proof}

We have the following corollary to Theorem \ref{chifomeg}.
\begin{corollary} \label{vertrans} Let $G$ be a a simple graph that is both vertex transitive and edge transitive
and suppose also that $\alpha(G)\omega(G)=n$.
Then $G$ is $k$-regular for some $k$ and $\Lmin(G)=-k/(\omega(G) -1).$
\end{corollary}

\begin{proof}
Because $G$ is edge transitive, each edge of $G$ is in the same number, $\mu$ say, of cliques of order $\omega=\omega(G)$.
Thus the set $\K$ of all cliques of order $\omega$ in $G$ is a clique partition of $\mu G$.
Because $G$ is vertex transitive, $\chi_f(G)=n/\alpha(G)$ \cite[p.142]{GandR}. Thus
$\chi_f(G)=\omega(G)$ and the corollary follows from Theorem \ref{chifomeg}.
\end{proof}

\begin{example}\label{composition}  {\rm ({\em Graph compositions}) Examples of simple graphs that satisfy the conditions of
Corollary \ref{vertrans} are the even cycles $C_{2m}$, the complete graphs $K_n$, the empty
graphs $K_n^c$, and the direct product $K_m\x K_n$.  If $G^1$ and $G^2$ are simple graphs, the {\em composition} $G^1[G^2]$
is the graph on the vertex set $V(G^1)\x V(G^2)$ with $(u,x)\sim (v,y)$ if either $u\sim v$ or
$u=v$ and $x\sim y$.
It is straightforward to show that if $G^1$ and $G^2$ satisfy the conditions of Corollary \ref{vertrans}
then so does $G^1[G^2]$.
For example, $K_m[K_n^c]$ is the regular complete multipartite graph with $m$ vertex parts
of order $n$ and, as we have already seen (more generally) in Example \ref{ctemultipart},
it must have least eigenvalue $\Lmin = -k/(\omega-1) = -(mn-n)/(m-1)=-n$.

Of course, if $G^1$ has order $m$ and $G^2$ has order $n$, then the adjacency matrix of
$G^1[G^2]$ is $A(G^1)\otimes J_n + I_m\otimes A(G^2)$. Thus,
if $G^2$ is $k$-regular, then the eigenvalues of $G^1[G^2]$ are
$\lambda_i(G^2)$, $i=2,\ldots,n$, each with multiplicity $m$, and
$n\lambda_j(G^1)+k$, $1\leq j\leq m$.
}\end{example}

\begin{example} \label{triangle}  {\rm ({\em Triangulated plane graphs}) Let $G$ be a plane graph every face of which is a triangle,
including the outer face.  Then the set $\K$ of all 3-cliques formed by the faces of $G$ is a clique
partition of $2G$.  By Lemma \ref{deltbnd},
$\Lmin(G) \ge -\Delta(G)/2$.  If $G$ has $n$ vertices and $e$ edges, then we also have
$\Lmin(G) \le  -\Lmax(G)/(\chi(G)-1) \le -2e/3n$ by (\ref{Lovasz}) and
the Four Color Theorem.  Thus, if $G$ is a triangulated $k$-regular plane graph with $\chi(G)=3$,
then $\Lmin(G)=-k/2$. For example, if $G=K_{2,2,2}$  is the octahedral graph, then
$\Lmin(G)=-2$.  However, if $G$ is the icosahedral graph, then $\Lmin(G)= -\sqrt{5} \approx -2.236$
(the same least eigenvalue as the dodecahedral graph),
but the best estimate we can obtain using clique partitions is $\Lmin(G) \ge -\Delta/2 =-2.5$.

If $G$ contains no $K_4$'s and is $k$-regular, then $\omega(G)=3$ and so,
by Corollary \ref{vertrans}, $\Lmin(G) = -k/2$
if $G$ is vertex-transitive and $\alpha=n/3$.
}\end{example}

The next two examples present something of a challenge.
Perhaps better lower bounds on $\Lmin(G)$ can be found using Theorem \ref{Dthm}.

\begin{example}{\rm ({\em The Shrikhande graph})  The {\em Shrikhande graph}
$G$ is a strongly regular graph with parameters
$(16,6,2,2)$ and so has least eigenvalue $\Lmin=-2$. But $\omega(G)=3$ and $G$ yields a triangulation
of the torus \cite[p. 21]{Biggs}, so, by Theorem \ref{chifomeg}, the best lower bound on
$\Lmin(G)$ that can be obtained using clique partitions of multiples of $G$
is only $\LminK(G)=-\Delta(G)/2=-3$.
}\end{example}

\begin{example}{\rm ({\em Kneser graphs})
The {\em Kneser graph} $\Kn(v,k)$ is the graph whose vertices are the
$k$-subsets of a $v$-set $X$.  Two vertices are adjacent if their
intersection is empty.  If $G=\Kn(v,k)$, then $\chi_f(G)=v/k$.  If $v=mk$
for some $m$, then partitions of $X$ into $k$-subsets are simple cliques of
order $m=v/k$, so $\omega(G)=\chi_f(G)$. Taking all such simple cliques yields a
clique partition of $\mu \Kn(v,k)$ with constant $r_u$, where
$$\mu = \frac{(mk-2k)!}{(k!)^{m-2}(m-2)!}$$
and
$$r_u= \frac{(mk-k)!}{(k!)^{m-1}(m-1)!}.$$
Thus by Theorem \ref{chifomeg},
$\Lmin=\LminK=-\binom{mk-k-1}{k-1}$.
The complete set of eigenvalues of $\Kn(v,k)$ is computed in \cite[Ch. 6]{GM} or \cite[Sec. 9.4]{GandR}.}\end{example}


\section{Smallest eigenvalue of $K_{1,k}$-free graphs}\label{sec:K1kfree}

A {\em claw free} graph $G$ is a graph that does not contain
$K_{1,3}$ as an induced subgraph. An equivalent formulation is that
for each vertex $x\in V(G)$, the neighbours of $x$ induce a subgraph
with independence number at most $2$. To quote from a paper of
Chudnovsky and Seymour \cite{CS}, {\em line graphs are claw-free, and it has
long been recognized that claw-free graphs are an interesting
generalization of line graphs, sharing some of their properties}. The eigenvalues of line graphs are at least $-2$. Linial \cite{Lin}
asked if the property of the eigenvalues of line graphs being bounded
below by an absolute constant is also true for regular claw free
graphs. In \cite{C0}, the first author showed that the answer is negative by describing a
family of regular claw-free graphs with arbitrarily negative eigenvalues. For sake of completeness, we briefly describe these examples here.
If $(\Gamma,+)$ is an additive finite group and $S$ is a symmetric subset of $\Gamma$ ($s\in S$ implies $-s\in S$) such that $0\notin S$, the Cayley graph $\Cay(\Gamma,S)$ has the elements of $\Gamma$ as vertices with $x$ adjacent to $y$ if and only if $x-y\in S$. Note
that $\Cay(\Gamma,S)$ is an undirected $d$-regular graph, where $d=|S|$. The eigenvalues of Abelian Cayley graphs $G=\Cay(\Gamma,S)$ can be easily expressed in terms of the irreducible characters of the group $\Gamma$. See
Li \cite{LiAbelian} for a proof and more details.
\begin{lemm}\label{abelian}
If $\Gamma$ is an Abelian group and $S$ a symmetric $d$-subset of
elements of $\Gamma$, then the eigenvalues of $\Cay(\Gamma,S)$ are
$\theta_{\chi}=\displaystyle\sum_{s\in S}\chi(s)$, where $\chi$ ranges over the characters of $\Gamma$.
\end{lemm}
Let $C_{n,r}$ be the graph with vertex set $\mathbb{Z}_{n}$ having $x$
adjacent to $y$ if and only if $x-y\in S_{r}\pmod{n}$, where
$S_{r}=\{\pm 1,\pm 2,\dots, \pm r\}$. This graph is the Cayley graph
of $\mathbb{Z}_{n}$ with generating set $S_r$ and is a
$2r$-regular graph. It is easy to see that $C_{n,r}$ is claw-free since the
neighborhood of each vertex of $C_{n,r}$ contains two disjoint
cliques of order $r$ and thus, it has independence number at most
$2$. Using Lemma \ref{abelian}, we can now calculate the eigenvalues
of $C_{n,r}$.
\begin{proposition}
The nontrivial eigenvalues of $C_{n,r}$ are $$-1+\frac{\sin\left((2r+1)\frac{\pi \ell}{n}\right)}{\sin \frac{\pi \ell}{n}},$$ for  $1 \le \ell \le n-1$.
\end{proposition}
\begin{proof}
For any $n$-th root of unity $\epsilon_{\ell}=e^{\frac{2\pi i\ell}{n}}$, the character of $\mathbb{Z}_n$ associated with it is $\chi(s)=\epsilon_{\ell}^s$ for $s\in \{0,\dots,n-1\}$. Lemma \ref{abelian} implies that for $1\leq \ell\leq n-1$,
\begin{align*}
\theta_{\chi}&=\sum_{j=1}^r\epsilon_{\ell}^j+\sum_{j=1}^r\epsilon_{\ell}^{-j}=\frac{1-\epsilon_{\ell}^{r+1}}{1-\epsilon_{\ell}}-1+\frac{1-\epsilon_{\ell}^{-(r+1)}}{1-\epsilon_{\ell}^{-1}}-1\\
&=-2+\frac{1-\epsilon_{\ell}^{r+1}-\epsilon_{\ell}(1-\epsilon_{\ell}^{-r-1})}{1-\epsilon_{\ell}}=-1+\frac{\epsilon_{l}^{-r}-\epsilon_{l}^{r+1}}{1-\epsilon_{l}}\\
&=-1+\frac{\epsilon_{\ell}^{r+\frac{1}{2}}-\epsilon_{\ell}^{r-\frac{1}{2}}}{\epsilon_{\ell}^{\frac{1}{2}}-\epsilon_{\ell}^{-\frac{1}{2}}}=-1+\frac{\sin\left((2r+1)\frac{\pi \ell}{n}\right)}{\sin \frac{\pi \ell}{n}}.
\end{align*}\end{proof}
If we choose $n$ and $r$ such that $\ell=\frac{3n}{2(2r+1)}$ is an
integer, then the previous proposition implies
\begin{equation*}
\Lmin(C_{n,r})\leq -1-\frac{1}{\sin\frac{3\pi}{2(2r+1)}}\sim
-1-\frac{2}{3\pi}-\frac{2}{3\pi}2r.
\end{equation*}
Hence, the eigenvalues of the claw-free graphs $C(n,r)$ can be arbitrarily negative.

Recently, Aharoni, Alon and Berger \cite{AAB} studied the largest eigenvalue of the Laplacian of $K_{1,k}$-free graphs for $k\geq 3$. A graph is $K_{1,k}$-free if it has no induced $K_{1,k}$. When $k=3$, this is the same as being claw-free. For simplicity, we will state the results from \cite{AAB} for regular graphs and in terms of their smallest adjacency eigenvalue. For $d\geq k\geq 3$, let $t(d,k)$ denote the maximum number of edges in a graph of order $d$ whose independence number is $k-1$ or less. By Tur\'{a}n's theorem, $t(d,k)$ equals the number of edges of a graph of order $d$ whose vertex set is partitioned into $k-1$ cliques, each of order $\lfloor \frac{d}{k-1}\rfloor$ or $\lceil \frac{d}{k-1}\rceil$.
\begin{theor}[\cite{AAB}]\label{aabthm}
If $G$ is a $d$-regular connected graph that is $K_{1,k}$-free, then
\begin{equation}
\Lmin(G)\geq -d+\frac{t(d,k)}{d-1}.
\end{equation}
\end{theor}
We note here that the same argument from \cite{AAB} can be used to prove a more general lower bound for $\Lmin(G)$.
\begin{proposition}\label{tm}
If $G$ is a connected $d$-regular graph where each vertex is contained in at least $m$ triangles and each edge is contained in at most $t$ triangles, then
\begin{equation}
\Lmin(G)\geq -d+\frac{m}{t}.
\end{equation}
\end{proposition}
\begin{proof}
The proof is the same as in \cite{AAB}, but for the sake of completeness, we describe it here. Take an eigenvector $\vx$ of length $1$ corresponding to $\Lmin$. It is easy to see that $d+\Lmin=\sum_{ij\in E}(x_i+x_j)^2$. Let $T$ be the set of triangles of $G$. For any edge $ij$, let $t_{ij}$ denote the number of triangles containing $ij$ and for any vertex $\ell$, let $t_{\ell}$ denote the number of triangles containing $\ell$. Clearly, $t_{ij}\leq t$ for any edge $ij$ and $t_{\ell}\geq m$ for any vertex $\ell$. Summing up the entries of $(x_i+x_j)^2+(x_j+x_{\ell})^2+(x_{\ell}+x_i)^2$ over all the triangles $ij\ell$ of $G$, we get that
\begin{align*}
\sum_{T:ij\ell}(x_i+x_j)^2+(x_j+x_{\ell})^2+(x_{\ell}+x_i)^2&=\sum_{ij\in E}t_{ij}(x_i+x_j)^2\\
&\leq t\sum_{ij\in E}(x_i+x_j)^2=t(d+\Lmin).
\end{align*}
On the other hand, for any triangle with vertices $i,j,\ell$, we have that:
\begin{align*}
(x_i+x_j)^2+(x_j+x_{\ell})^2+(x_{\ell}+x_i)^2&=x_i^2+x_j^2+x_{\ell}^2+(x_i+x_j+x_{\ell})^2\\
&\geq x_i^2+x_j^2+x_{\ell}^2,
\end{align*}
and therefore,
\begin{align*}
\sum_{T:ij\ell}(x_i+x_j)^2+(x_j+x_{\ell})^2+(x_{\ell}+x_i)^2&\geq \sum_{T:ij\ell}(x_i^2+x_j^2+x_{\ell}^2)\\
&=\sum_{i\in V}t_ix_i^2\geq m\sum_{i\in V}x_i^2\\
&=m.
\end{align*}
Combining these inequalities, we get that $t(d+\Lmin)\geq m$ which gives the desired result
\end{proof}
To see how one gets Theorem \ref{aabthm} from Proposition \ref{tm}, note that a $d$-regular $K_{1,k}$-free graph will have the property that $t\leq d-1$ and $m\geq t(d,k)$. We believe that the inequality in Theorem \ref{aabthm} may be improved although we have not been able to do so. Note that the hypothesis of $K_{1,k}$-free is not fully used in the proof of the previous theorem, but only its corollary that each neighborhood is dense. Perhaps for $k=3$, structural results on claw-free graphs like those in \cite{CS} may be used to improve these bounds. When $d=k=3$, Theorem \ref{aabthm} implies that a cubic claw-free graph $G$ must have $\Lmin(G)\geq -2.5$. We slightly improve this bound as follows.
\begin{theor}
Let $G$ be a connected $3$-regular claw-free graph on $n\geq 6$ vertices. Then
\begin{equation}
\Lmin(G)\geq \theta\approx -2.272,
\end{equation}
where $\theta$ is the smallest root of $x^3+x+14$.
\end{theor}
\begin{proof}
Let $\Lmin=\Lmin(G)$. Because $G$ is claw-free and cubic on $n\geq 6$ vertices, the neighborhood of each vertex is either $K_1\cup K_2$ or $K_{1,2}$. If all the neighborhoods are $K_1\cup K_2$, then $\Lmin\geq -2$. To see this, consider the edge-partition of $G$ into triangles and edges (such partition is unique in this case) and denote by $N$ its vertex-clique incidence matrix. Then $A(G)=NN^{T}+2I$ which gives the bound above.

Otherwise, if there are $K_{1,2}$ neighborhoods, then the graph will contain induced subgraphs on $4$ vertices consisting of $K_4$ minus one edge. We call such subgraphs diamonds and note that distinct diamonds will be vertex disjoint. If $H$ is a diamond with vertex set $\{a,b,u,v\}$, where $u$ and $v$ have degree $3$, then we call the edge $uv$ the middle edge of the diamond $H$. Let $\mathcal{M}$ be the set of middle edges of $G$.

We consider now the covering of the edges of $G$ by triangles where we used both triangles of each diamond (and where the triangles involved in a diamond cover the middle edge twice) and the triangles not involved in diamonds and the remaining edges. Let $M$ be the vertex-clique incidence matrix. Then
$$
MM^{T}=A+2I+B
$$
where $B$ is the adjacency matrix of the union of the disjoint middle edges and the remaining isolated vertices.

If $x$ is a unit eigenvector corresponding to $\Lmin$, then
\begin{equation}
\Lmin=x^{T}Ax=(M^{T}x)^{T}(M^{T}x)-2-x^{T}Bx
\end{equation}
which implies that
\begin{equation}\label{xuxv}
\Lmin\geq -2-x^{T}Bx=-2-2\sum_{uv\in \mathcal{M}}x_ux_v.
\end{equation}
We find an upper bound for $\sum_{uv\in \mathcal{M}}x_ux_v$ as follows. First, note that $x_u=x_v$ for every middle edge $uv$. To see this, consider the eigenvalue-eigenvector equation for the vertices $u$ and $v$ (where $a$ and $b$ are the other vertices involved in this diamond):
\begin{align*}
\Lmin x_u&=x_v+x_a+x_b\\
\Lmin x_v&=x_u+x_a+x_b.
\end{align*}
Therefore, $(\Lmin-1)x_u=(\Lmin-1)x_v=x_a+x_b$. This implies that
\begin{equation}\label{xaxb}
(\Lmin-1)^2x_u^2=(x_a+x_b)^2\leq 2(x_a^2+x_b^2).
\end{equation}
Combining \eqref{xuxv} and \eqref{xaxb} with the facts that $x$ has length one and that distinct diamonds are vertex disjoint, we will get
\begin{align*}
2\sum_{uv\in \mathcal{M}}x_ux_v&\leq \frac{4\sum_{a,b\sim u\in \mathcal{M}}(x_a^2+x_b^2)}{(\Lmin-1)^2}\\
&\leq \frac{4\left(1-2\sum_{uv\in \mathcal{M}}x_ux_v\right)}{(\Lmin-1)^2}.
\end{align*}
If $S=2\sum_{uv\in \mathcal{M}}x_ux_v$, then $S\leq \frac{4(1-S)}{(\Lmin-1)^2}$ which implies that $S\leq \frac{4}{4+(\Lmin-1)^2}$. Plugging this into \eqref{xuxv}, we get that
\begin{equation}
\Lmin\geq -2-\frac{4}{4+(\Lmin-1)^2}
\end{equation}
which implies that $\Lmin^3-\Lmin+14\geq 0$. Hence, $\Lmin\geq \theta\approx -2.272$.
\end{proof}
It may be possible that this type of argument can be extended for higher degrees in the case of quasi-line graphs (a special case of claw-free graphs where each neighborhood is a union of two cliques), but we leave this for a future work.

\section*{Acknowledgments}

The first author is grateful to the two anonymous referees, Francesco Belardo, Jack Koolen, Leonardo Lima and Krystal Guo for many useful comments and suggestions.

\end{document}